\pdfoutput=1
\RequirePackage{ifpdf}
\ifpdf 
\documentclass[pdftex]{sigma}
\else
\documentclass{sigma}
\fi

\numberwithin{equation}{section}

\newtheorem{Theorem}{Theorem}[section]
\newtheorem{Lemma}[Theorem]{Lemma}
{ \theoremstyle{definition}
\newtheorem{Definition}[Theorem]{Definition}
\newtheorem{Example}[Theorem]{Example}}

\begin{document}


\newcommand{\arXivNumber}{1407.6020}

\renewcommand{\PaperNumber}{082}

\FirstPageHeading

\ShortArticleName{Equivariant Join and Fusion of Noncommutative Algebras}

\ArticleName{Equivariant Join and Fusion\\ of Noncommutative Algebras}

\Author{Ludwik {D\c{A}BROWSKI}~$^\dag$, Tom {HADFIELD}~$^\ddag$  and Piotr M.~{HAJAC}~$^\S$}

\AuthorNameForHeading{L.~D\c{a}browski, T.~Hadf\/ield and P.M.~Hajac}

\Address{$^\dag$~SISSA (Scuola Internazionale Superiore di Studi Avanzati), \\
\hphantom{$^\dag$}~Via Bonomea 265, 34136 Trieste, Italy}
\EmailD{\href{mailto:dabrow@sissa.it}{dabrow@sissa.it}}

\Address{$^\ddag$~G-Research, Whittington House, 19-30 Alfred Place, London WC1E 7EA, UK}
\EmailD{\href{mailto:Thomas.Daniel.Hadfield@gmail.com}{Thomas.Daniel.Hadfield@gmail.com}}

\Address{$^\S$~Institytut Matematyczny, Polska Akademia Nauk,\\
\hphantom{$^\S$}~ul.~\'Sniadeckich 8, 00-656 Warszawa, Poland}
\EmailD{\href{mailto:pmh@impan.pl}{pmh@impan.pl}}
\URLaddressD{\url{http://www.impan.pl/~pmh/}}

\ArticleDates{Received June 30, 2015, in f\/inal form October 03, 2015; Published online October 13, 2015}

\Abstract{We translate the concept of the join of topological spaces to the language of \mbox{$C^*$-algebras},
replace the $C^*$-algebra of functions on the interval $[0,1]$ with evaluation maps at~$0$ and~$1$
by a unital $C^*$-algebra $C$ with appropriate two surjections, and introduce the notion of the \emph{fusion}
 of unital $C^*$-algebras.
An appropriate modif\/ication of this construction yields
the fusion comodule algebra of a comodule algebra $P$ with the coacting Hopf algebra~$H$.
We prove that, if the comodule
algebra $P$ is principal, then so is the fusion comodule algebra.
When $C=C([0,1])$ and the two surjections are evaluation maps at $0$ and~$1$, this result
 is a noncommutative-algebraic incarnation of the fact that, for a compact Hausdorf\/f principal $G$-bundle
$X$, the diagonal action of $G$ on the join $X*G$   is free.}


\Classification{46L85; 58B32}

\section{Introduction}

The join of topological spaces is a crucial concept in algebraic topology~--
it is used in the celebrated Milnor's construction
of a universal principal bundle~\cite{m-j56}.  The goal of this paper is to provide
a~noncomutative-geometric version
 of the join of a Cartan compact Hausdorf\/f principal bundle (no local triviality assumed)
with its structure group. In particular, we obtain  this way a noncommutative version
of the   diagonal $G$-action  on the $n$-fold join $G*\cdots*G$
of a compact Hausdorf\/f topological group $G$, which is the
f\/irst step in Milnor's construction.

To make this paper self-contained and to establish notation and terminology, we begin by recalling
the basics of classical joins, strong connections \cite{h-pm96}
and principal comodule algebras~\cite{hkmz11}.
In the f\/irst section, we def\/ine the join and fusion
$C^*$-algebra  of an arbitrary pair of unital \mbox{$C^*$-algebras}. For commutative $C^*$-algebras our join
construction recovers
Milnor's def\/inition in the case of compact Hausdorf\/f topological spaces. In the second section, we deal with a
join and fusion of a comodule algebra $P$ with the coacting
Hopf algebra~$H$.  To def\/ine an $H$-comodule algebra structure that  corresponds
to the diagonal action in the classical setting,
  we provide a~``gauged'' algebraic version of the join and fusion $C^*$-algebras.
(The gauged $C^*$-algebraic versions of the join and fusion are given respectively in~\cite{bdh} and~\cite{ddhw}.)
 The~main result of this paper is Theorem~\ref{main} concluding the principality of the fusion comodule algebra
of $P$ and $H$ from the principality
of~$P$.

\section{Classical join construction}

Let $I=[0,1]$ be the closed unit interval and let $X$ be a topological space. The \emph{unreduced suspension} $\Sigma X$
of $X$ is the quotient of $I\times X$ by the equivalence relation $R_S$ generated by
\begin{gather*}
(0,x)\sim (0,x'),\qquad (1,x)\sim (1,x').
\end{gather*}
Now take another topological space $Y$ and, on the space $I\times X\times Y$, consider the
 equivalence relation  $R_J$ given by
\begin{gather*}
(0,x,y)\sim (0,x',y),\qquad (1,x,y)\sim (1,x,y').
\end{gather*}
The quotient space
$X * Y:=(I\times X\times Y)/R_J$
is called the \emph{join} of $X$ and $Y$. It resembles the unreduced suspension of $X\times Y$,
but with only $X$ collapsed at 0, and only $Y$
collapsed at~1.
In particular, if $Y$ is a one-point space, the join $X*Y$ is  the {\em cone} $CX$ of~$X$.
If $Y$ is a two-point space with discrete topology, then the join $X*Y$ is the unreduced suspension~$\Sigma X$
of~$X$.
\begin{figure}[h]
\centering
\includegraphics[width=105mm]{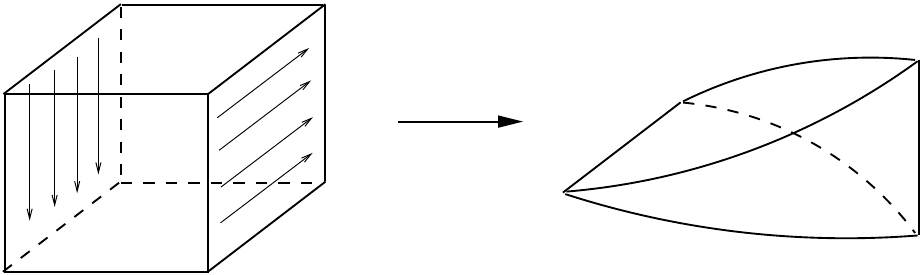}
\end{figure}

If $G$ is a topological group acting  continuously on  $X$ and $Y$ from the right,
then the diagonal right \mbox{$G$-action} on $X\times Y$ induces a  continuous action on the join $X*Y$.
 Indeed,
the diagonal action of $G$ on $I\times X\times Y$ factorizes to
the quotient, so that the formula
\begin{gather*}
([(t,x,y)],g)\longmapsto [(t,xg,yg)]
\end{gather*}
makes $X*Y$ a right $G$-space. It is immediate that this continuous action is free if the $G$-actions on $X$ and $Y$ are free.

Consider $CX\times Y$,  \mbox{$X\times CY$} and $X\times Y$ as $G$-spaces with the diagonal $G$-actions.
Note that there is a continuous surjection
$
\pi_I\colon X*Y \ni[(t,x,y)] \mapsto t\in [0,1]
$
such that $\pi_I^{-1}([0,\frac{1}{2}])$ and $\pi_I^{-1}([\frac{1}{2}, 1])$
are $G$-equivariantly homeomorphic  to $CX\times Y$ and \mbox{$X\times CY$} respectively.
Thus  $X*Y$ is $G$-equivariantly homeomorphic to the gluing  of $CX\times Y$ and \mbox{$X\times CY$}
over $X\times Y$:
\begin{gather*}
X*Y \cong (CX\times Y) \underset{X\times Y}{\sqcup} (X\times CY).
\end{gather*}

If $Y=G$ with the right action assumed to be the group multiplication,
we can construct the join $G$-space $X * Y$ in a  dif\/ferent manner: at 0 we collapse $X \times G$ to $G$ as
before,
and at~1 we collapse $X \times G$  to $(X\times G)/R_D$ instead of~$X$. Here $R_D$ is the equivalence
relation generated by
$(x,h)\sim (x',h')$, where $xh=x'h'$.
More precisely, let
 $R_J'$ be the  equivalence relation on $I \times X\times G$ generated by
\begin{gather}\label{5}
(0,x,h)\sim (0,x',h)\qquad \text{and}\qquad (1,x,h)\sim (1,x',h'), \qquad \text{where}\quad xh=x'h'.
\end{gather}
The formula $[(t,x,h)]k:=[(t,x,hk)]$ def\/ines a continuous right $G$-action on $(I \times X\times G)/R_J'$.
One can also easily check that the formula{\samepage
\begin{gather}\label{6}
X*G\ni [(t,x,h)] \longmapsto \big[\big(t,xh^{-1},h\big)\big] \in (I\times X\times G)/R_J'
\end{gather}
yields a $G$-equivariant homeomorphism.}

If we further specify also $X=G$ with the right action assumed to be the group multiplication,
  then the $G$-action on $X*Y=G * G$ is automatically free.  Furthermore, since the  action of~$G$ on
 $X*G$ is free whenever it is free on~$X$,
we conclude that the natural action on the iterated join of~$G$ with itself is also free.
For instance, for $G={\mathbb Z}/2{\mathbb Z}$ we obtain a~${\mathbb Z}/2{\mathbb Z}$-equivariant identif\/ication
$({\mathbb Z}/2{\mathbb Z})^{*(n+1)}\cong S^{n}$,
where $S^{n}$ is the $n$-dimensional sphere with the antipodal action of~${\mathbb Z}/2{\mathbb Z}$.

\section{Strong connection and principal comodule algebra}\label{sc}

Let $H$ be a Hopf algebra with coproduct~$\Delta$, counit~$\varepsilon$ and bijective
antipode~$S$. Next, let \mbox{$\delta\colon P\to P\otimes H$} be a coaction making $P$ a right
\mbox{$H$-comodule} algebra.
We shall frequently use
the Heyneman--Sweedler notation (with the summation sign suppressed) for coproduct and coaction:
\begin{gather*}
\Delta (h)=:h_{(1)}\otimes h_{(2)} ,\qquad
\delta(p)=:p_{(0)}\otimes p_{(1)} .
\end{gather*}
\begin{Definition}[\cite{hkmz11}]
Let $P$ be a right comodule algebra over a Hopf algebra
$H$ with bijective antipode, and let
\begin{gather*}
B:=P^{\operatorname{co}H} := \{ p \in P \, | \, \delta (p) = p \otimes 1 \}.
\end{gather*}
be the coaction-invariant
subalgebra. The comodule algebra $P$
is called {\em principal} if the following conditions are satisf\/ied:
\begin{enumerate}\itemsep=0pt
\item[1)]
the coaction of $H$ is Hopf--Galois, that is, the map
\begin{gather*}
\operatorname{can}_P \colon \  P \underset{B}{\otimes} P \longrightarrow P \otimes H,\qquad
	p \otimes q \longmapsto pq_{(0)} \otimes q_{(1)} ,
\end{gather*}
(called the canonical map) is bijective,
\item[2)]
the comodule algebra $P$ is right
$H$-equivariantly projective as a left  $B$-module, i.e.,
 there exists
 a right $H$-colinear and left $B$-linear
splitting of the multiplication map
$B \otimes P \rightarrow P$.
\end{enumerate}
\end{Definition}

\begin{Definition}[\cite{bh04}]\label{strong}
Let $ H$ be a Hopf algebra with  bijective antipode.
A \emph{strong connection $\ell$} on a right $H$-comodule algebra $P$ is a unital linear map $\ell\colon  H
\rightarrow  P \otimes  P$ satisfying
the following bicolinearity (i.e., left and right colinearity) and splitting conditions:
\begin{enumerate}\itemsep=0pt
\item[1)]
$(\operatorname{id}\otimes \delta) \circ
\ell = (\ell \otimes \operatorname{id}) \circ \Delta$,
$(\delta^L \otimes \operatorname{id}) \circ
\ell = (\operatorname{id} \otimes \ell) \circ
\Delta$, where
$\delta^L :=(S^{-1}\otimes\operatorname{id})\circ\mathrm{f\/lip}\circ\delta$;
\item[2)]
$\widetilde{\operatorname{can}_P} \circ \ell=1 \otimes \operatorname{id}$, where
$\widetilde{\operatorname{can}_P}\colon  P\otimes  P\ni p\otimes q\mapsto (p\otimes 1)\delta(q)\in  P\otimes H$.
\end{enumerate}
\end{Definition}

We will use the  Heyneman--Sweedler-type notation
$
\ell(h)=:\ell(h)^{\langle 1 \rangle}\otimes
 \ell(h)^{\langle 2 \rangle}
$
 with the summation sign suppressed.
 One can easily prove (see \cite[p.~599]{hkmz11} and references therein) that
a~comodule algebra is principal if and only if it admits a strong
connection.
Here we need the following slight generalization of this fact:
\begin{Lemma}\label{general}
Let $H$ be a Hopf algebra with bijective antipode. Then a right $H$-comodule algebra~$P$ is principal if and
only if it admits a~$($not necessarily unital$)$ linear map $\ell \colon H \rightarrow  P \otimes  P$ satisfying the
bicolinearity and splitting conditions of Definition~{\rm \ref{strong}}.
\end{Lemma}

\begin{proof}
Assume that $\ell \colon H \rightarrow  P \otimes  P$ is a linear map satisfying the bicolinearity and splitting
conditions.
Let $\pi_B\colon P\otimes P\to P\otimes_BP$ be the canonical surjection. Def\/ine
\begin{gather*}
L\colon \ P\otimes H\ni p\otimes h\longmapsto \pi_B\big(p \ell(h)^{\langle 1\rangle}\otimes\ell(h)^{\langle 2\rangle}\big)\in P\underset{B}{\otimes}P.
\end{gather*}
It follows immediately from the  splitting property of $\ell$ that
\begin{gather*}
\operatorname{can}_P(L(p\otimes h))= p \, \widetilde{\operatorname{can}_P}(\ell(h))=p\otimes h.
\end{gather*}
Applying $\operatorname{id}\otimes\varepsilon$ to the splitting condition for $\ell$, we obtain $m\circ\ell=\varepsilon$, where $m$ is
the multiplication map on~$P$.
Combining it with the left colinearity of $\ell$, which implies that
\begin{gather*}
\delta\big(q_{(0)}\ell(q_{(1)})^{\langle 1\rangle}\big)\otimes\ell(q_{(1)})^{\langle 2\rangle}
=q_{(0)}\ell(q_{(3)})^{\langle 1\rangle}\otimes q_{(1)}S(q_{(2)})\otimes\ell(q_{(3)})^{\langle 2\rangle}\\
\hphantom{\delta\big(q_{(0)}\ell(q_{(1)})^{\langle 1\rangle}\big)\otimes\ell(q_{(1)})^{\langle 2\rangle}}{}
=q_{(0)}\ell(q_{(1)})^{\langle 1\rangle}\otimes 1\otimes\ell(q_{(1)})^{\langle 2\rangle},
\end{gather*}
we obtain
\begin{gather*}
L(\operatorname{can}_P(p\otimes q)) = \pi_B \big(pq_{(0)} \ell(q_{(1)})^{\langle 1\rangle}
 \otimes  \ell(q_{(1)})^{\langle 2\rangle}\big)
 = \pi_B \big(p \otimes q_{(0)} \ell(q_{(1)})^{\langle 1\rangle} \ell(q_{(1)})^{\langle 2\rangle}\big)
 \\
\hphantom{L(\operatorname{can}_P(p\otimes q))}{}
=  \pi_B(p\otimes q).
\end{gather*}
Finally, the bicolinearity of $\ell$ implies that the formula $s(p):= p_{(0)}\ell(p_{(1)})$ def\/ines a left $B$-linear
right $H$-colinear splitting
of the multiplication map $B\otimes P\to P$. The reverse implication (principality $\Rightarrow$ existence of a~bicolinear~$\ell$ with the splitting property)
follows from~\cite[Lemma~2.2]{bh04}.
\end{proof}

\section[Join and fusion of unital $C^*$-algebras]{Join and fusion of unital $\boldsymbol{C^*}$-algebras}

Let $\otimes_{\min}$ denote the spatial (minimal) tensor product.
Due to the fact that minimal tensor products preserve injections
 (e.g., see \cite[Proposition~4.22]{t-m79} or
\cite[Section~1.3]{w-s94}),
for any unital $C^*$-algebras $A_1$ and $A_2$ the natural maps
\begin{gather*}
A_1\ni a\longmapsto a\otimes 1\in A_1 \underset{\min}{\otimes} A_2\qquad \text{and}
\qquad  A_2\ni a\longmapsto 1\otimes a\in A_1 \underset{\min}{\otimes} A_2
\end{gather*}
are injective. This allows us to view $A_1$ and $A_2$ as subalgebras of
$A_1 \otimes_{\min} A_2$. We use this natural identif\/ication in what follows.
\begin{Definition}\label{gjoin}
Let $A_1$ and $A_2$  be unital $C^*$-algebras, and
let $C$ be a unital $C^*$-algebra equipped with two $C^*$-ideals $J_1$ and $ J_2$
such that the direct sum
$
\pi_1\oplus\pi_2\colon C\rightarrow (C/J_1) \oplus (C/J_2)
$
of the canonical surjections $\pi_i\colon C\to C/J_i$, $i\in \{1,2\}$, is surjective.
We call the unital $C^*$-algebra
\begin{gather*}
A_1 \underset{\pi}{\circledast} A_2 :=
\big\{x\in C \underset{\min}{\otimes} A_1 \underset{\min}{\otimes} A_2 \,
\big|\,
(\pi_i\otimes \operatorname{id})(x) \in
(C/J_i)\otimes_{\min}A_i
, \, i\in\{1,2\}
\big\}
\end{gather*}
the \emph{fusion} $C^*$-algebra of $A_1$ and~$A_2$ over~$C$.
When $C$ is the $C^*$-algebra $C([0,1])$ of all continuous complex-valued
functions on the unit interval and
$\pi_1:=\mathrm{ev}_1$ and $\pi_2:=\mathrm{ev}_0$ are
the evaluation maps respectively at $1$ and $0$,
we call $A_1 \circledast_\pi A_2$ the \emph{join} $C^*$-algebra of $A_1$ and~$A_2\,$, and denote it by
 $A_1 \circledast A_2$.
\end{Definition}

When $C$ is
a commutative $C^*$-algebra, the projections $\pi_i$  correspond to inclusions  of
two disjoint closed subsets (playing the role of points $0$ and $1$) into a compact Hausdorf\/f
space (playing the role of~$[0,1]$).
Observe also that,
if $A_1:=C(X)$ and $A_2 := C(Y)$ are the \mbox{$C^*$-algebras}
 of continuous functions on compact Hausdorf\/f spaces
$X$ and $Y$ respectively,
then
\begin{gather*}
A_1 \circledast A_2 = C(X * Y).
\end{gather*}
On the noncommutative side,
it is worth mentioning  that the join of matrix algebras is used to construct the celebrated
Jiang--Su $C^*$-algebra~\cite{js99}.

\section[Join and fusion of a principal $H$-comodule algebra with $H$]{Join and fusion of a principal $\boldsymbol{H}$-comodule algebra with $\boldsymbol{H}$}

Since a diagonal coaction is not in general an algebra homomorphism,
to obtain an equivariant version of our noncommutative fusion and join constructions,
we need to modify  Def\/inition~\ref{gjoin}
in the spirit of~\eqref{5},~\eqref{6}.
To avoid analytical complications that
are tackled in~\cite{ddhw}, we also need to change the setting from $C^*$-algebraic to algebraic. In particular,
we replace $\otimes_{\min}$ by~$\otimes$.

\begin{Definition}\label{pjoin}
Let $\delta\colon P\to P\otimes H$ be a right coaction making $P$ a comodule algebra, and
let~$C$ be a unital $C^*$-algebra equipped with two $C^*$-ideals $J_1$ and~$ J_2$
such that the direct sum
\begin{gather*}
\pi_1\oplus\pi_2\colon \ C\longrightarrow (C/J_1) \oplus (C/J_2)
\end{gather*}
of the canonical surjections $\pi_i\colon C\to C/J_i$, $i\in \{1,2\}$, is surjective. We call the algebra
\begin{gather*}
P \underset{\pi}{\overset{\delta}{\circledast}} H :=
\big\{x\in C \otimes P \otimes H \,
\big|\, \\
\hphantom{P \underset{\pi}{\overset{\delta}{\circledast}} H :=\big\{}{}
(\pi_1\otimes \operatorname{id})(x) \in (C/J_1)\otimes \delta (P) ,\,
(\pi_2\otimes \operatorname{id})(x) \in (C/J_2) \otimes \mathbb{C} \otimes H\big\}
\end{gather*}
the \emph{equivariant fusion} algebra of $P$ and~$H$ over~$C$.
When $C:=C([0,1])$  and
$\pi_1:=\mathrm{ev}_1$, $\pi_2:=\mathrm{ev}_0$,
we call $P \circledast^\delta_\pi H$ the \emph{equivariant join} algebra of $P$ and~$H$,
and denote it by $P \circledast^\delta H$.
\end{Definition}

\begin{Lemma}
The coaction $\operatorname{id}\otimes\operatorname{id}\otimes\Delta\colon
C\otimes P\otimes H\to C\otimes P\otimes H\otimes H$
restricts and corestricts to $\Delta_{P\circledast{}^\delta_\pi H}\colon P{\circledast}^\delta_\pi   H
\to (P{\circledast}^\delta_\pi H)\otimes H$ making $P{\circledast}^\delta_\pi H$
a right $H$-comodule algebra.
\end{Lemma}

\begin{proof}
Note f\/irst that for $i\in\{1,2\}$, we have
\begin{gather*}
\big((\pi_i\otimes\operatorname{id}\otimes\operatorname{id})\circ(\operatorname{id}\otimes\operatorname{id}\otimes\Delta)\big)\!\bigg(\sum_j f_j\otimes p_j\otimes h_j\bigg)
=(\operatorname{id}\otimes\operatorname{id}\otimes\Delta)\!\bigg(\sum_j \pi_i(f_j) \otimes p_j\otimes h_j\bigg).
\end{gather*}
Assume now that $\sum_j f_j\otimes p_j\otimes h_j\in P\circledast^\delta_\pi H$.
 Then the above tensor belongs to
\begin{gather*}
(C/J_1)\otimes \big((\operatorname{id}\otimes\Delta)(\delta(P))\big)=
(C/J_1)\otimes \big((\delta\otimes\operatorname{id})(\delta(P))\big)
\subseteq(C/J_1)\otimes \delta(P)\otimes H
\end{gather*}
for $i=1$ and to $(C/J_2)\otimes \mathbb{C}\otimes H\otimes H$ for~$i=2$.
\end{proof}

The main result of this paper is:
\begin{Theorem}\label{main}
Let $C$ be a unital $C^*$-algebra equipped with two $C^*$-ideals $J_1$ and $ J_2$
such that the direct sum
$
\pi_1\oplus\pi_2\colon C\rightarrow (C/J_1) \oplus (C/J_2)
$
of the canonical surjections $\pi_i\colon C\to C/J_i$, $i\in \{1,2\}$, is surjective.
Then, for any principal right $H$-comodule algebra~$P$,
 the equivariant fusion right $H$-comodule algebra $P{\circledast}^\delta_\pi H$ is principal.
\end{Theorem}

\begin{proof}
Note f\/irst that, due to  the surjectivity of $\pi_1\oplus\pi_2$, we can choose $x\in C$  such that
\begin{gather*}
(\pi_1\oplus\pi_2)(x)=(1,0).
\end{gather*}
Since $\{0, 1\}\subseteq \operatorname{spec}(x^*x)$, we can
 take a continuous function $f\colon \operatorname{spec}(x^*x)\to [0,1]$
such that $f(0)=0$ and $f(1)=1$, and def\/ine $t:=f(x^*x)\in C$.
Then there exist elements $\sqrt{t},\sqrt{1-t}\in C$ enjoying the properties
$\pi_2(\sqrt{t})=0=\pi_1(\sqrt{1-t})$.

Next, let
$\ell \colon H{\to} P\otimes P$  be a strong connection on $P$.
Then the bicolinearity of~$\ell$ implies that the linear map
\begin{gather*}
\tilde\ell \colon \  H \longrightarrow \big( C \otimes P \otimes H\big) \otimes \big( C \otimes P \otimes H\big),\\
\tilde\ell(h):= {\sqrt t} \otimes  \ell (h_{(2)})^{\langle 1\rangle}\otimes S(h_{(1)})
\otimes {\sqrt t} \otimes  \ell (h_{(2)})^{\langle 2\rangle}  \otimes h_{(3)}
\\
\hphantom{\tilde\ell(h):=}{}+ \sqrt{1-t} \otimes 1 \otimes S(h_{(1)}) \otimes \sqrt{1-t}
\otimes 1\otimes h_{(2)},
\end{gather*}
corestricts to $(P\circledast^\delta_\pi H)\otimes  (P\circledast^\delta_\pi H)$.
Indeed,
\begin{gather*}
 (\pi_1\otimes\operatorname{id}\otimes\operatorname{id})\otimes(\operatorname{id}\otimes\operatorname{id}\otimes\operatorname{id})(\widetilde{\ell}(h))\\
 \qquad{} =
\big([1]\otimes\delta\big(\ell(h_{(1)})^{\langle 1\rangle}\big)\big)\otimes
\big(\sqrt{t}\otimes \ell(h_{(1)})^{\langle 2\rangle}\otimes h_{(2)}\big)
\in \big(C/J_1\otimes\delta(P)\big)\otimes\big(C\otimes P\otimes H\big),\\
 (\pi_2\otimes\operatorname{id}\otimes\operatorname{id})\otimes(\operatorname{id}\otimes\operatorname{id}\otimes\operatorname{id})(\widetilde{\ell}(h))\\
 \qquad{} =
\big([1]\otimes 1\otimes S(h_{(1)}\big)\otimes
\big(\sqrt{1-t}\otimes 1\otimes h_{(2)}\big)
\in \big(C/J_2\otimes\mathbb{C}\otimes H\big)\otimes\big(C\otimes P\otimes H\big).
\end{gather*}
Much in the same way,
\begin{gather*}
 (\operatorname{id}\otimes\operatorname{id}\otimes\operatorname{id})\otimes(\pi_1\otimes\operatorname{id}\otimes\operatorname{id})(\widetilde{\ell}(h))
\in \big(C\otimes P\otimes H\big)\otimes\big(C/J_1\otimes\delta(P)\big),\\
 (\operatorname{id}\otimes\operatorname{id}\otimes\operatorname{id})\otimes(\pi_2\otimes\operatorname{id}\otimes\operatorname{id})(\widetilde{\ell}(h))
\in\big(C\otimes P\otimes H\big) \otimes\big(C/J_2\otimes\mathbb{C}\otimes H\big).
\end{gather*}

The bicolinearity of the thus corestricted $\tilde\ell$ is evident. Finally, taking
advantage of the right colinearity and the splitting property of~$\ell$, we check that
\begin{gather*}
 \widetilde{\operatorname{can}_{P\circledast^\delta_\pi H}}(\tilde\ell (h)) =
 t \otimes  \ell (h_{(2)})^{\langle 1\rangle}\ell (h_{(2)})^{\langle 2\rangle}
\otimes S(h_{(1)})  h_{(3)}\otimes h_{(4)}\\
\hphantom{\widetilde{\operatorname{can}_{P\circledast^\delta_\pi H}}(\tilde\ell (h)) =}{}
+ (1-t) \otimes 1 \otimes S(h_{(1)})h_{(2)} \otimes h_{(3)}
 \\
\hphantom{\widetilde{\operatorname{can}_{P\circledast^\delta_\pi H}}(\tilde\ell (h))}{} =  t \otimes  \ell (h_{(2)})^{\langle 1\rangle}\ell (h_{(2)})^{\langle 2\rangle}\!{}_{(0)} \otimes S(h_{(1)})
\ell (h_{(2)})^{\langle 2\rangle}\!{}_{(1)}\otimes h_{(3)}
+ (1-t) \otimes 1 \otimes 1 \otimes h
 \\
\hphantom{\widetilde{\operatorname{can}_{P\circledast^\delta_\pi H}}(\tilde\ell (h))}{} =  t \otimes  1 \otimes S(h_{(1)})  h_{(2)}\otimes h_{(3)}
+ (1-t) \otimes 1 \otimes 1 \otimes h
 \\
\hphantom{\widetilde{\operatorname{can}_{P\circledast^\delta_\pi H}}(\tilde\ell (h))}{}
 =  1 \otimes  1 \otimes 1 \otimes h.
\end{gather*}
Now the claim follows from Lemma~\ref{general}.
\end{proof}

Observe that in the special case of the join construction, we can take $t$ to be the inclusion map
$[0,1]\to\mathbb{C}$. Moreover, as explained in the next section, the equivariant join
comodule algebra $P\circledast^\delta H$
becomes piecewise trivial, and Theorem~\ref{main} follows from \cite[Lemma~3.2]{hkmz11}.
However, even in this special case of the equivariant join comodule algebra,
our non-unital strong-connection formula is dif\/ferently constructed than the complicated strong-connection
 formula in~\cite[Lemma~3.2]{hkmz11}. This might be very important for index pairing computations involving
concrete strong-connection formulas.
Finally, let us remark that a fully f\/ledged \mbox{$C^*$-algebraic} approach to Theorem~\ref{main}
can be found in~\cite{ddhw}.

\begin{Example}
In particular, taking $P=H$ to be the Hopf algebra of ${\rm SU}_q(2)$, we can conclude
the principality of Pf\/laum's noncommutative instanton
bundle~\cite{p-mj94}. The fact that this bundle is not trivializable at the $C^*$-algebraic level
 was proved in \cite{bdah} using~\cite{dhhmw12}.
In the classical setting,
except for the trivial group,
all joins $G*G$, where $G$ is a compact Hausdorf\/f group, are non-trivial as principal $G$-bundles. We
conjecture the same is true for all compact quantum groups~\cite{w-sl98}.
\end{Example}

\section{Piecewise structure}

This section is an adjustment of \cite[Section~5.1]{hkmz11} to the setting of  equivariant join construction.
It is important to bear in mind that we cannot expect to have such a piecewise (pullback) structure in the
general case of equivariant fusion comodule algebra.

Let $P$ be a principal right $H$-comodule algebra with a coaction~$\delta$.
First we def\/ine the following algebras
\begin{gather*}
  P_1:=\left\{x\in C([0,1]) \otimes P \otimes H \,
\big|\, (\mathrm{ev}_0\otimes \operatorname{id}\otimes \operatorname{id})(x) \in \mathbb{C}\otimes H) \right\},\nonumber\\
  P_2:=\left\{x\in C([0,1]) \otimes P \otimes H \,
\big|\, (\mathrm{ev}_1\otimes \operatorname{id}\otimes \operatorname{id})(x) \in \delta (P) \right\}.
\end{gather*}
The $P_i$'s become
$H$-comodule algebras for the coactions obtained by the restrictions and corestrictions of
\mbox{$\operatorname{id}\otimes \operatorname{id} \otimes \Delta$},
and the subalgebras of $H$-coaction invariants are respectively
\begin{gather*}
  B_1:=\{y \in C([0,1])\otimes P \,|\, (\mathrm{ev}_0\otimes \operatorname{id})(y) \in \mathbb{C}\},
		  \nonumber\\
  B_2:=\{y\in C([0,1])\otimes P \,|\, (\mathrm{ev}_1\otimes \operatorname{id})(y) \in P^{\operatorname{co}H}\}.
\end{gather*}
 Now one can identify $P$ with the pullback comodule algebra
\begin{gather*}
	\{(p,q) \in P_1 \oplus P_2\,|\,
	(\mathrm{ev}_1\otimes \operatorname{id})(p)=(\mathrm{ev}_0\otimes \operatorname{id})(q)\}
\end{gather*}
of the $P_i$'s along the right $H$-colinear algebra homomorphisms
\begin{gather*}
(\mathrm{ev}_1\otimes \operatorname{id})\colon \  P_1 \longrightarrow P\otimes H,\qquad
(\mathrm{ev}_0\otimes \operatorname{id})\colon \ P_2 \longrightarrow P\otimes H.
\end{gather*}

\subsection*{Acknowledgements}

All authors are grateful to Piotr M.~So{\l}tan and  Karen R.~Strung for references concerning
 the minimal tensor pro\-duct and the Jiang--Su $C^*$-algebra respectively.
 Ludwik D\c{a}browski and Piotr M.~Hajac  were  partially supported
by  PRIN 2010-11 grant ``Operator Algebras,
 Noncommutative Geometry and
Applications'' and NCN grant 2011/01/B/ST1/06474, respectively.
Tom Hadf\/ield was f\/inanced via the EU Transfer of Know\-ledge contract
MKTD-CT-2004-509794.  Also, Piotr M.~Hajac is very thankful to SISSA for its hospitality.

\pdfbookmark[1]{References}{ref}
\LastPageEnding

\end{document}